\documentclass[11pt,dvipsnames,a4paper]{article}
\usepackage[affil-it]{authblk}
\usepackage[utf8]{inputenc}
\usepackage{amsmath, amssymb, mathrsfs, amsthm, amsfonts}
\usepackage{latexsym}
\usepackage{hyperref}
\usepackage{a4wide}
\usepackage{color,xcolor}
\usepackage{graphicx}
\usepackage{fancyhdr}

\newtheorem{theorem}{Theorem}[section]

\newtheorem{lemma}[theorem]{Lemma}
\newtheorem{proposition}[theorem]{Proposition}

\theoremstyle{definition}

\theoremstyle{remark}

\title{Completely Positive Maps for Imprimitive Reflection Groups}

\author{Hery Randriamaro
	\thanks{This research was funded by my mother \\
		Lot II B 32 bis Faravohitra, 101 Antananarivo, Madagascar \\
		e-mail: \texttt{hery.randriamaro@gmail.com}}}

\begin{document}

\thispagestyle{empty}

\noindent {\color{MidnightBlue} \rule{\linewidth}{2pt}}

\vspace*{15pt}

\noindent \textsl{\Huge Completely Positive Maps for Imprimitive Reflection Groups}

\vspace*{25pt}

\noindent \textbf{\Large Hery Randriamaro${}^1$} 

\vspace*{10pt}

\noindent ${}^1$\: This research was funded by my mother \\ \textsc{Lot II B 32 bis Faravohitra, 101 Antananarivo, Madagascar} \\ \texttt{hery.randriamaro@gmail.com}

\vspace*{25pt}

\noindent \textsc{\large Abstract} \\
This article proves the existence of completely positive quasimultiplicative maps from the group algebra of imprimitive reflection groups to the set of bounded operators, and uses those linear maps to define creation and annihilation operators on the full Fock space.

\vspace*{10pt}


\vspace*{10pt}


\vspace*{10pt}

\noindent {\color{MidnightBlue} \rule{\linewidth}{2pt}}

\vspace*{10pt}

\section{Introduction}

\noindent Recall that, for $m,n \in \mathbb{N}^*$, the complex reflection group $G(m,1,n)$ is generated by reflections $s_0, s_1, s_2, \dots, s_{n-1}$ on $\mathbb{C}^n$ subject to the relations
\begin{align*}
& s_0^m = s_1^2 = s_2^2 = \dots = s_{n-1}^2 = 1,
& s_0 s_1 s_0 s_1 = s_1 s_0 s_1 s_0, \\
& \forall i \in [n-2]:\, s_i s_{i+1} s_i = s_{i+1} s_i s_{i+1},
& \forall i,j \in [0,n-1],\, |i-j| \geq 2:\, s_i s_j = s_j s_i.
\end{align*}

\noindent The study of cyclotomic descent algebras by Mathas and Orellana \cite[§~2]{MaOr} leads to consider the reflections $t_1 := s_0$, and $t_{i+1} := s_i t_i s_i$ for $i \in [n-1]$. The subgroup $T := \langle t_1, t_2, \dots, t_n \rangle$ is isomorphic to $(\mathbb{Z}/m\mathbb{Z})^n$, and is normal in $G(m,1,n)$. Since $\langle s_1, s_2, \dots, s_{n-1} \rangle$ is the symmetric group $\mathrm{Sym}(n)$ of order $n$, then $G(m,1,n) = T \rtimes \mathrm{Sym}(n)$. As set, we have $$G(m,1,n) = \big\{t_1^{\alpha_1} t_2^{\alpha_2} \dots t_n^{\alpha_n} w\ \big|\ \alpha_1, \alpha_2, \dots, \alpha_n \in [0,m-1],\, w \in \mathrm{Sym}(n)\big\}.$$
Hence, the set $\varPi = \{t_1, t_2, \dots, t_n, s_1, s_2, \dots, s_{n-1}\}$ generates $G(m,1,n)$ so that Mathas and Orellana could define the length function $\mathrm{l}: G(m,1,n) \rightarrow \mathbb{N}$ given by \cite[Definition~2.3]{MaOr} $$\mathrm{l}(g) := \min \{k \in \mathbb{N}\ |\ \exists r_1, r_2, \dots, r_k \in \varPi:\, g = r_1 r_2 \dots r_k\}.$$

\noindent Let call $\varPi$ a set of representative reflections of $G(m,1,n)$. In addition to those above, the reflections in $\varPi$ are also subject to the relations $$\forall i,j \in [n],\, i \neq n,\, |i-j| \geq 2:\ s_i t_i s_i t_i = t_i s_i t_i s_i,\ s_i t_{i+1} s_i t_{i+1} = t_{i+1} s_i t_{i+1} s_i,\ s_i t_j = t_j s_i.$$

\noindent We work on a complex Hilbert space $\mathbf{H}$ endowed with an inner product $\langle \centerdot , \centerdot \rangle: \mathbf{H} \times \mathbf{H} \rightarrow \mathbb{C}$. Recall that $\mathbf{H}$ is also a normed space with norm $\|\centerdot\|: \mathbf{H} \rightarrow \mathbb{R}_+$ defined by $\|x\| := \sqrt{\langle x,x \rangle}$.

\noindent A bounded operator on $\mathbf{H}$ is a linear operator $\mathsf{L}: \mathbf{H} \rightarrow \mathbf{H}$ for which there exists a number $\alpha \in \mathbb{R}_+$ such that, for every $x \in \mathbf{H}$, $\|\mathsf{L} x\| \leq \alpha \|x\|$. The set $\mathscr{B}_{\mathbf{H}}$ of bounded operators on $\mathbf{H}$ is a normed algebra with norm $\|.\|_{\mathrm{op}}$ defined, for every bounded operator $\mathsf{L}$, by $$\|\mathsf{L}\|_{\mathrm{op}} := \inf \big\{\alpha \in \mathbb{R}_+\ \big|\ \forall x \in \mathbf{H}:\, \|\mathsf{L} x\| \leq \alpha \|x\|\big\}.$$

\noindent The adjoint of $\mathsf{L} \in \mathscr{B}_{\mathbf{H}}$ is the linear operator $\mathsf{L}^*: \mathbf{H} \rightarrow \mathbf{H}$ such that $\langle \mathsf{L}^*x,y \rangle = \langle x, \mathsf{L}y \rangle$ for every $x,y \in \mathbf{H}$. Recall that $\mathsf{L}$ is said self-adjoint if $\mathsf{L}^* = \mathsf{L}$.

\smallskip

\noindent Let $\displaystyle f = \sum_{u \in G(m,1,n)} f(u)u$ and $\displaystyle g = \sum_{u \in G(m,1,n)} g(u)u$ belong to $\mathbb{C}G(m,1,n)$. The group algebra $\mathbb{C}G(m,1,n)$ is also a Hilbert space for the inner product $\langle \centerdot , \centerdot \rangle: \mathbb{C}G(m,1,n) \times \mathbb{C}G(m,1,n) \rightarrow \mathbb{C}$ defined by $\displaystyle \langle f,g \rangle := \sum_{u \in G(m,1,n)} \overline{f(u)} g(u)$. Moreover, as the map $g \mapsto fg$ is a linear operator on $\mathbb{C}G(m,1,n)$, we may define the adjoint of $f$ by $\displaystyle f^* := \sum_{u \in G(m,1,n)} \overline{f(u^{-1})}\,u \in \mathbb{C}G(m,1,n)$.

\noindent A linear map $\varphi: \mathbb{C}G(m,1,n) \rightarrow \mathscr{B}_{\mathbf{H}}$ is quasimultiplicative if $\varphi(1) = \mathrm{id}_{\mathbf{H}}$ and $$\forall u,v \in G(m,1,n):\ \mathrm{l}(uv) = \mathrm{l}(u) + \mathrm{l}(v) \, \Rightarrow \, \varphi(uv) = \varphi(u) \varphi(v).$$

\noindent And a linear map $\varphi: \mathbb{C}G(m,1,n) \rightarrow \mathscr{B}_{\mathbf{H}}$ is completely positive if $$\forall k \in \mathbb{N}^*,\, f_1, \dots, f_k \in \mathbb{C}G(m,1,n),\, x_1, \dots, x_k \in \mathbf{H}:\ \Big\langle \sum_{i,j=1}^k \varphi(f_j^*\,f_i)x_i,\, x_j\Big\rangle \geq 0.$$

\begin{theorem} \label{ThPo}
Let $\varPi = \{t_1, t_2, \dots, t_n, s_1, s_2, \dots, s_{n-1}\}$ be a representative reflection set of the complex reflection group $G(m,1,n)$, and $\mathsf{S}_1, \dots, \mathsf{S}_{n-1}, \mathsf{T}_1, \dots, \mathsf{T}_n \in \mathscr{B}_{\mathbf{H}}$ such that
\begin{itemize}
\item $\forall i \in [n-1]:\ \|\mathsf{S}_i\|_{\mathrm{op}} \leq 1,\, \mathsf{S}_i^* = \mathsf{S}_i$,
\item $\displaystyle \forall j \in [n]:\ \|\mathsf{T}_j\|_{\mathrm{op}} \leq 1,\, \sum_{i \in [m-1]} (\mathsf{T}_j^i)^* = \sum_{i \in [m-1]} \mathsf{T}_j^i$,
\item the bounded operators satisfy the following braid relations:
\begin{align*}
& \forall i \in [n-2]:\ \mathsf{S}_i \mathsf{S}_{i+1} \mathsf{S}_i = \mathsf{S}_{i+1} \mathsf{S}_i \mathsf{S}_{i+1}, \quad \forall i,j \in [n-1],\, |i-j| \geq 2:\ \mathsf{S}_i \mathsf{S}_j = \mathsf{S}_j \mathsf{S}_i, \\
& \forall i,j \in [n]:\ \mathsf{T}_i \mathsf{T}_j = \mathsf{T}_j \mathsf{T}_i, \quad \forall i,j \in [n],\, i \neq n,\, |i-j| \geq 2:\ \mathsf{S}_i \mathsf{T}_j = \mathsf{T}_j \mathsf{S}_i, \\
& \forall i \in [n-1]:\ \mathsf{S}_i \mathsf{T}_i \mathsf{S}_i \mathsf{T}_i = \mathsf{T}_i \mathsf{S}_i \mathsf{T}_i \mathsf{S}_i,\ \mathsf{S}_i \mathsf{T}_{i+1} \mathsf{S}_i \mathsf{T}_{i+1} = \mathsf{T}_{i+1} \mathsf{S}_i \mathsf{T}_{i+1} \mathsf{S}_i.
\end{align*}
\end{itemize}
Then, the quasimultiplicative linear map $\varphi: \mathbb{C}G(m,1,n) \rightarrow \mathscr{B}_{\mathbf{H}}$ given by $$\varphi(1) = \mathrm{id}_{\mathbf{H}}, \quad \forall i \in [n-1]:\, \varphi(s_i) = \mathsf{S}_i \quad \text{and} \quad \forall j \in [n]:\, \varphi(t_j) = \mathsf{T}_j$$ is completely positive.
\end{theorem}

\noindent Take some vector $\varOmega \in \mathbf{H}$ with $\|\varOmega\| = 1$ called vacuum. The full Fock space for $\mathbf{H}$ is $\displaystyle \bigoplus_{n \in \mathbb{N}} \mathbf{H}^{\otimes n}$ where $\mathbf{H}^0 := \mathbb{C}\varOmega$. Assume that we are given some operators $\mathsf{S}, \mathsf{T}$ such that
\begin{itemize}
	\item $\mathsf{S} \in \mathscr{B}_{\mathbf{H} \otimes \mathbf{H}}$, $\|\mathsf{S}\|_{\mathrm{op}} < 1$, $\mathsf{S}^* = \mathsf{S}$, and $\mathsf{S}$ fulfills the braid relation also called Yang-Baxter equation $(\mathrm{id}_{\mathbf{H}} \otimes \mathsf{S}) (\mathsf{S} \otimes \mathrm{id}_{\mathbf{H}}) (\mathrm{id}_{\mathbf{H}} \otimes \mathsf{S}) = (\mathsf{S} \otimes \mathrm{id}_{\mathbf{H}}) (\mathrm{id}_{\mathbf{H}} \otimes \mathsf{S}) (\mathsf{S} \otimes \mathrm{id}_{\mathbf{H}})$ \cite[§~8.1]{Ka},
	\item $\mathsf{T} \in \mathscr{B}_{\mathbf{H}}$, $\|\mathsf{T}\|_{\mathrm{op}} < 1$, $\displaystyle \sum_{i \in [m-1]} (\mathsf{T}^i)^* = \sum_{i \in [m-1]} \mathsf{T}^i$, and $\mathsf{T}$ fulfills the relations
	$$(\mathsf{T} \otimes \mathrm{id}_{\mathbf{H}}) \mathsf{S} (\mathsf{T} \otimes \mathrm{id}_{\mathbf{H}}) \mathsf{S} = \mathsf{S} (\mathsf{T} \otimes \mathrm{id}_{\mathbf{H}}) \mathsf{S} (\mathsf{T} \otimes \mathrm{id}_{\mathbf{H}}),$$
	$$(\mathrm{id}_{\mathbf{H}} \otimes \mathsf{T}) \mathsf{S} (\mathrm{id}_{\mathbf{H}} \otimes \mathsf{T}) \mathsf{S} = \mathsf{S} (\mathrm{id}_{\mathbf{H}} \otimes \mathsf{T}) \mathsf{S} (\mathrm{id}_{\mathbf{H}} \otimes \mathsf{T}).$$
\end{itemize}

\noindent Let $n \in \mathbb{N}^*$, $i \in [n-1]$, and $j \in [n]$. Define the operators $\mathsf{S}_i, \mathsf{T}_j \in \mathscr{B}_{\mathbf{H}^{\otimes n}}$ by
$$\mathsf{S}_i := \overbrace{\mathrm{id}_{\mathbf{H}} \otimes \dots \otimes \mathrm{id}_{\mathbf{H}}}^{i-1\ \text{times}} \otimes \mathsf{S} \otimes \overbrace{\mathrm{id}_{\mathbf{H}} \otimes \dots \otimes \mathrm{id}_{\mathbf{H}}}^{n-i-1\ \text{times}}\ \text{and}\ \mathsf{T}_j := \overbrace{\mathrm{id}_{\mathbf{H}} \otimes \dots \otimes \mathrm{id}_{\mathbf{H}}}^{j-1\ \text{times}} \otimes \mathsf{T} \otimes \overbrace{\mathrm{id}_{\mathbf{H}} \otimes \dots \otimes \mathrm{id}_{\mathbf{H}}}^{n-j\ \text{times}}.$$
The $\mathsf{S}_i$'s and $\mathsf{T}_j$'s fulfill the assumptions of Theorem~\ref{ThPo}. Define $\mathsf{P}_n: \mathbf{H}^{\otimes n} \rightarrow \mathbf{H}^{\otimes n}$ as the operator given by
$$\mathsf{P}_0 := \mathrm{id}_{\mathbb{C}\varOmega} \quad \text{and} \quad \mathsf{P}_n := \sum_{u \in G(m,1,n)} \varphi(u),$$
where $\varphi$ is the quasimultiplicative map of Theorem~\ref{ThPo} with $\varphi(s_i) = \mathsf{S}_i$ and $\varphi(t_j) = \mathsf{T}_j$.

\begin{theorem} \label{ThQu}
The sesquilinear form $\displaystyle \langle \centerdot , \centerdot \rangle_G: \bigoplus_{n \in \mathbb{N}} \mathbf{H}^{\otimes n} \times \bigoplus_{n \in \mathbb{N}} \mathbf{H}^{\otimes n} \rightarrow \mathbb{C}$ given by
$$\forall X \in \mathbf{H}^{\otimes k},\, \forall Y \in \mathbf{H}^{\otimes n}:\ \langle X,Y \rangle_G := \delta_{kn} \langle X,\, \mathsf{P}_n Y \rangle$$ is an inner product so that, for each $x \in \mathbf{H}$, one can define creation and annihilation operators $\displaystyle \mathsf{d}^*(x): \bigoplus_{n \in \mathbb{N}} \mathbf{H}^{\otimes n} \rightarrow \bigoplus_{n \in \mathbb{N}} \mathbf{H}^{\otimes n}$ and $\displaystyle \mathsf{d}(x): \bigoplus_{n \in \mathbb{N}} \mathbf{H}^{\otimes n} \rightarrow \bigoplus_{n \in \mathbb{N}} \mathbf{H}^{\otimes n}$ respectively which are adjoint with respect to $\langle \centerdot , \centerdot \rangle_G$. Moreover, we have $\displaystyle \|\mathsf{d}^*(x)\|_{\mathrm{op}} \leq \frac{\|x\|}{\sqrt{\big(1- \|\mathsf{S}\|_{\mathrm{op}}\big) \big(1- \|\mathsf{T}\|_{\mathrm{op}}\big)}}$.
\end{theorem}

\noindent We use the proof strategy of Bo\.{z}ejko and Speicher to prove in Section~\ref{SeTh1} the complete positivity of the quasimultiplicative linear map $\varphi$ in Theorem~\ref{ThPo}, and to build in Section~\ref{SeTh2} the inner product, the creator, and the annihilator in Theorem~\ref{ThQu}.

\section{Completely Positive Maps} \label{SeTh1}

\noindent We prove Theorem~\ref{ThPo} in this section. Remark that we recover completely positive maps on $\mathrm{Sym}(n)$ \cite[Theorem~1.1]{BoSp} by letting $\mathsf{T}_j = 0$ in that theorem.

\begin{lemma} \label{LeInv}
Let $\varPi = \{t_1, t_2, \dots, t_n, s_1, s_2, \dots, s_{n-1}\}$ be a representative reflection set of the complex reflection group $G(m,1,n)$, and $\varphi$ the quasimultiplicative map of Theorem~\ref{ThPo}. If $\|\varphi(r)\|_{\mathrm{op}} < 1$ for every $r \in \varPi$, then the operator $\displaystyle \mathsf{P} := \sum_{u \in G(m,1,n)} \varphi(u) \in \mathscr{B}_{\mathbf{H}}$ is invertible.
\end{lemma}

\begin{proof}
We have $\displaystyle \mathsf{P} = \sum_{t \in T} \varphi(t) \sum_{w \in \mathrm{Sym}(n)} \varphi(w)$. On one side, Bo\.{z}ejko and Speicher already proved that $\displaystyle \sum_{w \in \mathrm{Sym}(n)} \varphi(w)$ is invertible \cite[Theorem~2.4]{BoSp}. On the other side,
$$\sum_{t \in T} \varphi(t) = \prod_{j \in [n]} \sum_{i \in [0,m-1]} \varphi(t_j)^i.$$
Remark that $\displaystyle \big(\mathrm{id}_{\mathbf{H}} - \varphi(t_j)\big)\sum_{i \in [0,m-1]} \varphi(t_j)^i = \mathrm{id}_{\mathbf{H}} - \varphi(t_j)^m$. Since $\|\varphi(t_j)^m\|_{\mathrm{op}} < \|\varphi(t_j)\|_{\mathrm{op}} < 1$, both operators $\mathrm{id}_{\mathbf{H}} - \varphi(t_j)^m$ and $\mathrm{id}_{\mathbf{H}} - \varphi(t_j)$ are consequently invertible. Hence the operator $\displaystyle \sum_{i \in [0,m-1]} \varphi(t_j)^i$ is invertible with inverse $\displaystyle \big(\mathrm{id}_{\mathbf{H}} - \varphi(t_j)\big) \sum_{k \in \mathbb{N}} \varphi(t_j)^{mk}$.
\end{proof}

\noindent An operator $\mathsf{L} \in \mathscr{B}_{\mathbf{H}}$ is said positive and strictly positive if, respectively, $\langle \mathsf{L}x,x \rangle \geq 0$ and $\langle \mathsf{L}x,x \rangle > 0$ for every nonzero vector $x \in \mathbf{H}$.

\begin{lemma} \label{LeStr}
Let $\varPi = \{t_1, t_2, \dots, t_n, s_1, s_2, \dots, s_{n-1}\}$ be a representative reflection set of the complex reflection group $G(m,1,n)$, and $\varphi$ the quasimultiplicative map of Theorem~\ref{ThPo}. If $\|\varphi(r)\|_{\mathrm{op}} < 1$ for every $r \in \varPi$, then the operator $\displaystyle \mathsf{P} = \sum_{u \in G(m,1,n)} \varphi(u)$ is strictly positive.
\end{lemma}

\begin{proof}
Bo\.{z}ejko and Speicher defined for a self-adjoint operator $\mathsf{L} \in \mathscr{B}_{\mathbf{H}}$ the number $$\mathrm{m}_0(\mathsf{L}) := \inf \big\{\langle \mathsf{L}x,x \rangle \in \mathbb{R}\ \big|\ x \in \mathbf{H},\, \|x\| = 1\big\}$$ which is in fact the smallest element in the spectrum of $\mathsf{L}$, and proved that, for $\mathsf{L}, \mathsf{K} \in \mathscr{B}_{\mathbf{H}}$ self-adjoint, $\big|\mathrm{m}_0(\mathsf{L}) - \mathrm{m}_0(\mathsf{K})\big| \leq \|\mathsf{L} - \mathsf{K}\|_{\mathrm{op}}$ \cite[Lemma~2.5]{BoSp}.

\noindent If $0 \leq q \leq 1$, let $\varphi_q: \mathbb{C}G(m,1,n) \rightarrow \mathscr{B}_{\mathbf{H}}$ be the quasimultiplicative linear map such that, for every $r \in \varPi$, $\varphi_q(r) := q \varphi(r)$. The operator $\displaystyle \mathsf{P}_q := \sum_{u \in G(m,1,n)} \varphi_q(u) \in \mathscr{B}_{\mathbf{H}}$ is self-adjoint since
\begin{align*}
\mathsf{P}_q^* & = \sum_{w \in \mathrm{Sym}(n)} \varphi_q(w)^* \sum_{t \in T} \varphi_q(t)^* \\
& = \sum_{w \in \mathrm{Sym}(n)} \varphi_q(w^{-1}) \bigg(\prod_{j \in [n]} \sum_{i \in [0,m-1]} \varphi_q(t_j)^i\bigg)^* \\
& = \sum_{w \in \mathrm{Sym}(n)} \varphi_q(w) \bigg(\prod_{j \in [n]} \sum_{i \in [0,m-1]} \varphi_q(t_j)^i\bigg) \\
& = \sum_{w \in \mathrm{Sym}(n)} \varphi_q(w) \sum_{t \in T} \varphi_q(t) \\
& = \mathsf{P}_q.
\end{align*}
Then, the map $q \mapsto \mathrm{m}_0(\mathsf{P}_q)$ is continuous, since the map $q \mapsto \mathsf{P}_q$ is norm-continuous and $\big|\mathrm{m}_0(\mathsf{P}_{q_1}) - \mathrm{m}_0(\mathsf{P}_{q_2})\big| \leq \|\mathsf{P}_{q_1} - \mathsf{P}_{q_2}\|_{\mathrm{op}}$. Remark that $\mathsf{P}_0 = \mathrm{id}_{\mathbf{H}}$ and $\mathsf{P}_1 = \mathsf{P}$. The invertibility of $\mathsf{P}_q$ deduced from Lemma~\ref{LeInv} implies $\mathrm{m}_0(\mathsf{P}_q) \neq 0$. As $\mathrm{m}_0(\mathsf{P}_0) = 1$, we necessarily have $\mathrm{m}_0(\mathsf{P}_q) > 0$, in particular $\mathrm{m}_0(\mathsf{P}_1) > 0$.
\end{proof}

\begin{proposition} \label{PrPo}
Let $\varPi = \{t_1, t_2, \dots, t_n, s_1, s_2, \dots, s_{n-1}\}$ be a representative reflection set of the complex reflection group $G(m,1,n)$, and $\varphi$ the quasimultiplicative map of Theorem~\ref{ThPo}. If $\|\varphi(r)\|_{\mathrm{op}} \leq 1$ for every $r \in \varPi$, then the operator $\displaystyle \mathsf{P} = \sum_{u \in G(m,1,n)} \varphi(u)$ is positive.
\end{proposition}

\begin{proof}
The argument is similar to that of \cite[Theorem~2.2]{BoSp}. Let $\varphi_q: \mathbb{C}G(m,1,n) \rightarrow \mathscr{B}_{\mathbf{H}}$, for $0 \leq q < 1$, be the quasimultiplicative linear map in the proof of Lemma~\ref{LeStr}. We know from Lemma~\ref{LeStr} that $\displaystyle \mathsf{P}_q := \sum_{u \in G(m,1,n)} \varphi_q(u) \in \mathscr{B}_{\mathbf{H}}$ is strictly positive. As $\displaystyle \lim_{q \to 1^-} \mathsf{P}_q = \mathsf{P}$ uniformly, we get the result.
\end{proof}

\noindent For $v \in G(m,1,n)$, let $\theta_v: \mathbb{C}G(m,1,n) \rightarrow \mathbb{C}G(m,1,n)$ be the operator $$\displaystyle \theta_v(f) := vf= \sum_{u \in G(m,1,n)}f(v^{-1}u)u \quad \text{with norm} \quad \|\theta_v\|_{\mathrm{op}} = 1.$$
Define the quasimultiplicative linear map $\lambda: \mathbb{C}G(m,1,n) \rightarrow \mathscr{B}_{\mathbb{C}G(m,1,n)}$ given by $$\lambda(1) = \mathrm{id}_{\mathbb{C}G(m,1,n)}, \quad \forall i \in [n-1]:\, \lambda(s_i) = \theta_{s_i} \quad \text{and} \quad \forall j \in [n]:\, \lambda(t_j) = \theta_{t_j}.$$
One can easily verify that $\lambda(s_i)$ and $\displaystyle \sum_{i \in [m-1]} \lambda(t_j)^i$ are self-adjoint. 

\noindent We can now establish the proof of Theorem~\ref{ThPo}.

\begin{proof}
For $i \in [n-1]$ and $j \in [n]$, define the operators $\hat{\mathsf{S}}_i$ and $\hat{\mathsf{T}}_j$ on $\mathbb{C}G(m,1,n) \otimes \mathbf{H}$ respectively by $\hat{\mathsf{S}}_i := \lambda(s_i) \otimes \mathsf{S}_i$ and $\hat{\mathsf{T}}_j := \lambda(t_j) \otimes \mathsf{T}_j$. And let $\hat{\varphi}: \mathbb{C}G(m,1,n) \rightarrow \mathscr{B}_{\mathbb{C}G(m,1,n) \otimes \mathbf{H}}$ be the quasimultiplicative linear map given by $$\hat{\varphi}(1) = \mathrm{id}_{\mathbb{C}G(m,1,n) \otimes \mathbf{H}}, \quad \forall i \in [n-1]:\, \hat{\varphi}(s_i) = \lambda(s_i) \otimes \mathsf{S}_i \quad \text{and} \quad \forall j \in [n]:\, \hat{\varphi}(t_j) = \lambda(t_j) \otimes \mathsf{T}_j.$$ As $\|\hat{\varphi}(s_i)\|_{\mathrm{op}} \leq 1$ and $\|\hat{\varphi}(t_j)\|_{\mathrm{op}} \leq 1$, we deduce from Proposition~\ref{PrPo} that the operator $\displaystyle \hat{\mathsf{P}} = \sum_{u \in G(m,1,n)} \hat{\varphi}(u)$ is positive. Let $x_1, \dots, x_k \in \mathbf{H}$, $f_1, \dots, f_k \in \mathbb{C}G(m,1,n)$, and set $$\displaystyle y = \sum_{u \in G(m,1,n)} u \otimes \sum_{i \in [k]}f_i(u^{-1})x_i \in \mathbb{C}G(m,1,n) \otimes \mathbf{H}.$$
Then, on the Hilbert space $\mathbb{C}G(m,1,n) \otimes \mathbf{H}$, we have
\begin{align*}
0 \leq \langle \hat{\mathsf{P}}y,y \rangle & = \sum_{u \in G(m,1,n)} \Big\langle \big(\lambda(u) \otimes \varphi(u)\big)y \, , \, y \Big\rangle \\
& = \sum_{u,v,w \in G(m,1,n)} \sum_{i,j \in [k]} \Big\langle \big(\lambda(u) \otimes \varphi(u)\big)\big(v \otimes f_i(v^{-1})x_i \big) \, , \, w \otimes f_j(w^{-1})x_j     \Big\rangle \\
& = \sum_{u,v,w \in G(m,1,n)} \sum_{i,j \in [k]} \Big\langle \theta_{u}(v) \otimes f_i(v^{-1}) \, \varphi(u)x_i \, , \, w \otimes f_j(w^{-1})x_j \Big\rangle \\
& = \sum_{u,v,w \in G(m,1,n)} \sum_{i,j \in [k]} \langle uv, w \rangle \, \Big\langle f_i(v^{-1}) \, \varphi(u)x_i \, , \, f_j(w^{-1})x_j \Big\rangle \\
& = \sum_{u,v,w \in G(m,1,n)} \sum_{i,j \in [k]} \langle uv,w \rangle \, \Big\langle \overline{f_j(w^{-1})} f_i(v^{-1}) \, \varphi(u)x_i \, , \, x_j \Big\rangle \\
& = \sum_{u,v \in G(m,1,n)} \sum_{i,j \in [k]} \Big\langle \overline{f_j(v^{-1}u^{-1})} f_i(v^{-1}) \, \varphi(u)x_i \, , \, x_j \Big\rangle \\
& =  \sum_{i,j \in [k]} \bigg\langle \varphi\Big(\sum_{u,v \in G(m,1,n)} \overline{f_j(v^{-1})} f_i(v^{-1}u)  \,u\Big)x_i \, , \, x_j \bigg\rangle \\
& = \sum_{i,j \in [k]} \big\langle \varphi(f_j^*\,f_i)x_i,\, x_j \big\rangle.
\end{align*}
\end{proof}

\section{Representation on Fock Space} \label{SeTh2}

\noindent We prove Theorem~\ref{ThQu} in this section. The inner product defined by Bo\.{z}ejko and Speicher \cite[Theorem~3.1]{BoSp} is a special case of $\langle \centerdot , \centerdot \rangle_G$. If $\{e_i\}_{i \in I}$ is some basis of $\mathbf{H}$ and $q_{ij} \in \mathbb{C}$, they used it for the construction of the Fock representation of the $q_{ij}$-relations $$\mathsf{d}(e_i) \mathsf{d}^*(e_j) - q_{ij} \mathsf{d}^*(e_j) \mathsf{d}(e_i) = \delta_{ij}$$
where $q_{ji} = \bar{q}_{ij}$ and $|q_{ij}| \leq 1$. Meljanac and Svrtan independently did the same construction \cite[§~1.1]{MeSv}. That $q_{ij}$-relation model is a generalization of previously studied models \cite{Grb1}, \cite{Sp}, \cite{Za}.

\smallskip

\noindent Recall that the canonical free creation and annihilation operators $\displaystyle \mathsf{l}^*(x): \bigoplus_{n \in \mathbb{N}} \mathbf{H}^{\otimes n} \rightarrow \bigoplus_{n \in \mathbb{N}} \mathbf{H}^{\otimes n}$ and $\displaystyle \mathsf{l}(x): \bigoplus_{n \in \mathbb{N}} \mathbf{H}^{\otimes n} \rightarrow \bigoplus_{n \in \mathbb{N}} \mathbf{H}^{\otimes n}$ respectively, for each $x \in \mathbf{H}$, are given by \cite[§~2]{Ev} 
\begin{align*}
& \mathsf{l}^*(x)\varOmega = x \quad \text{and} \quad \mathsf{l}^*(x)\, x_1 \otimes \dots \otimes x_n = x \otimes x_1 \otimes \dots \otimes x_n, \\
& \mathsf{l}(x)\varOmega = 0 \, \ \quad \text{and} \quad \mathsf{l}(x)\, x_1 \otimes \dots \otimes x_n = \langle x, x_1 \rangle \, x_2 \otimes \dots \otimes x_n,
\end{align*}
for $x_1, \dots, x_n \in \mathbf{H}$. Consider the operator
$$\mathsf{R}_n := \Big(\mathrm{id}_{\mathbf{H}^{\otimes n}} + \sum_{i \in [m-1]} \mathsf{T}_1^i\Big) \, \Big(\mathrm{id}_{\mathbf{H}^{\otimes n}} + \sum_{j \in [n-1]} \prod_{i \in [j]}^{\rightarrow}  \mathsf{S}_i\Big) \in \mathscr{B}_{\mathbf{H}^{\otimes n}}.$$
We define the creation and annihilation operators $\mathsf{d}^*(x)$ and $\mathsf{d}(x)$ in Theorem~\ref{ThQu} by $$\mathsf{d}^*(x) := \mathsf{l}^*(x) \quad \text{and} \quad \mathsf{d}(x)\, x_1 \otimes \dots \otimes x_n := \mathsf{l}(x)\,\mathsf{R}_n\, x_1 \otimes \dots \otimes x_n.$$

\noindent Remark that they are clearly not adjoint with respect to the usual inner product $\langle \centerdot , \centerdot \rangle$.

\noindent We can now establish the proof of Theorem~\ref{ThQu}.

\begin{proof}
As stated in Lemma~\ref{LeStr}, $\mathsf{P}_n$ is strictly positive, then $\langle \centerdot , \centerdot \rangle_G$ is an inner product. By definition of $\mathsf{S}_i$ and $\mathsf{T}_j$, we have $\mathsf{l}^*(x) \mathsf{S}_i = \mathsf{S}_{i+1} \mathsf{l}^*(x)$ and $\mathsf{l}^*(x) \mathsf{T}_j = \mathsf{T}_{j+1} \mathsf{l}^*(x)$, which implies
$$\mathsf{l}^*(x) \mathsf{P}_n = (\mathrm{id}_{\mathbf{H}} \otimes \mathsf{P}_n) \mathsf{l}^*(x) \quad \text{or} \quad \mathsf{P}_n \mathsf{l}(x) =  \mathsf{l}(x) (\mathrm{id}_{\mathbf{H}} \otimes \mathsf{P}_n).$$
Note that $\displaystyle \mathsf{P}_{n+1} = \Big(\sum_{u \in G(m,1,n)} \varphi(u)\Big) \Big(\varphi(1) + \sum_{i \in [m]} \varphi(t_1^i) \Big) \Big(\varphi(1) + \sum_{j \in [n]} \prod_{i \in [j]}^{\rightarrow} \varphi(s_i)\Big)$, where the complex reflection group $G(m,1,n)$ is generated by $\{t_2, t_3, \dots, t_{n+1}, s_2, s_3, \dots, s_n\}$ as reflection set, $\displaystyle \sum_{u \in G(m,1,n)} \varphi(u) = \mathsf{P}_n$, and $\displaystyle \Big(\varphi(1) + \sum_{i \in [m]} \varphi(t_1^i) \Big) \Big(\varphi(1) + \sum_{j \in [n]} \prod_{i \in [j]}^{\rightarrow} \varphi(s_i)\Big) = \mathsf{R}_{n+1}$. Then,
$$\mathsf{P}_{n+1} = (\mathrm{id}_{\mathbf{H}} \otimes \mathsf{P}_n) \mathsf{R}_{n+1}.$$
Now, for $X \in \mathbf{H}^{\otimes n}$ and $Y \in \mathbf{H}^{\otimes n+1}$, we have
\begin{align*}
\big\langle\mathsf{d}^*(x)X,\, Y\big\rangle_G & = \big\langle\mathsf{d}^*(x)X,\, \mathsf{P}_{n+1}Y\big\rangle \\
& = \big\langle X,\, \mathsf{l}(x)\mathsf{P}_{n+1}Y\big\rangle \\
& = \big\langle X,\, \mathsf{l}(x)(\mathrm{id}_{\mathbf{H}} \otimes \mathsf{P}_n) \mathsf{R}_{n+1}Y\big\rangle \\
& = \big\langle X,\, \mathsf{P}_n \mathsf{l}(x) \mathsf{R}_{n+1} Y\big\rangle \\
& = \big\langle X,\, \mathsf{P}_n \mathsf{d}(x) Y\big\rangle \\
& = \big\langle X,\, \mathsf{d}(x) Y\big\rangle_G.
\end{align*}

\noindent Moreover, since
$$\|\mathsf{R}_{n+1}\|_{\mathrm{op}} \leq \Big(\sum_{i \in [0,n]} \|\mathsf{S}\|_{\mathrm{op}}^i\Big) \Big(\sum_{j \in [0,m-1]} \|\mathsf{T}\|_{\mathrm{op}}^j\Big) \leq \frac{1}{\big(1- \|\mathsf{S}\|_{\mathrm{op}}\big) \big(1- \|\mathsf{T}\|_{\mathrm{op}}\big)},$$
then 
$$\|\mathsf{P}_{n+1} \mathsf{P}_{n+1}\|_{\mathrm{op}} = \big\|(\mathrm{id}_{\mathbf{H}} \otimes \mathsf{P}_n) \mathsf{R}_{n+1} \mathsf{R}_{n+1}^* (\mathrm{id}_{\mathbf{H}} \otimes \mathsf{P}_n) \big\|_{\mathrm{op}} \leq \frac{\big\|(\mathrm{id}_{\mathbf{H}} \otimes \mathsf{P}_n) (\mathrm{id}_{\mathbf{H}} \otimes \mathsf{P}_n) \big\|_{\mathrm{op}}}{\big(1- \|\mathsf{S}\|_{\mathrm{op}}\big)^2\, \big(1- \|\mathsf{T}\|_{\mathrm{op}}\big)^2},$$
hence $\displaystyle \|\mathsf{P}_{n+1}\|_{\mathrm{op}} \leq \frac{\big\|(\mathrm{id}_{\mathbf{H}} \otimes \mathsf{P}_n) \big\|_{\mathrm{op}}}{\big(1- \|\mathsf{S}\|_{\mathrm{op}}\big) \big(1- \|\mathsf{T}\|_{\mathrm{op}}\big)}$. Denoting by $\|\centerdot\|_G$ the norm associated to $\langle \centerdot , \centerdot \rangle_G$, for $X \in \mathbf{H}^{\otimes n}$, we consequently have
\begin{align*}
\big\|\mathsf{d}^*(x)X\big\|_G & = \big\langle \mathsf{d}^*(x)X,\, \mathsf{d}^*(x)X\big\rangle_G \\
& = \langle x \otimes X,\, x \otimes X \rangle_G \\
& = \langle x \otimes X,\, \mathsf{P}_{n+1}\, x \otimes X \rangle \\
& \leq \frac{1}{\big(1- \|\mathsf{S}\|_{\mathrm{op}}\big) \big(1- \|\mathsf{T}\|_{\mathrm{op}}\big)} \big\langle x \otimes X,\, (\mathrm{id}_{\mathbf{H}} \otimes \mathsf{P}_n)\, x \otimes X \big\rangle \\
& = \frac{1}{\big(1- \|\mathsf{S}\|_{\mathrm{op}}\big) \big(1- \|\mathsf{T}\|_{\mathrm{op}}\big)} \langle x,x \rangle \langle X,\, \mathsf{P}_n X \rangle_G \\
& = \frac{\|x\|\, \|X\|_G}{\big(1- \|\mathsf{S}\|_{\mathrm{op}}\big) \big(1- \|\mathsf{T}\|_{\mathrm{op}}\big)}.
\end{align*}
\end{proof}

\bibliographystyle{abbrvnat}

\end{document}